\documentclass[11pt]{article}
\usepackage{enumerate}
\usepackage{amsmath}
\usepackage{amsthm}
\usepackage{amsfonts}
\usepackage{amssymb}
\usepackage[numbers]{natbib}
\usepackage{xcolor}
\usepackage{graphicx}

\usepackage{fullpage}
\usepackage{xcolor}
\usepackage{hyperref}
\usepackage{bbm}
\newcommand*\samethanks[1][\value{footnote}]{\footnotemark[#1]}
\newcommand{\keywords}[1]
{
	{\small\textbf{Keywords:} #1}
}

\usepackage{todonotes}

\newtheorem{question}{Question}
\numberwithin{question}{section}
\newtheorem{conjecture}[question]{Conjecture}
\newtheorem{problem}[question]{Problem}

\newtheorem{theorem}[question]{Theorem}
\newtheorem{proposition}[question]{Proposition}
\newtheorem{lemma}[question]{Lemma}

\newtheorem{definition}[question]{Definition}
\newtheorem{construction}[question]{Construction}

\newtheorem{remark}[question]{Remark}
\numberwithin{equation}{section}
\title{On density conditions for transversal trees in multipartite graphs}
\author{Leila Badakhshian\thanks{Email: l.badakhshian@gmail.com}\and Victor Falgas-Ravry\thanks{Institutionen f\"or Matematik och Matematisk Statistik, Ume{\aa} Universitet, 901 87 Ume{\aa}, Sweden. Emails: \texttt{victor.falgas-ravry}, \texttt{maryam.sharifzadeh} .\texttt{@umu.se}. Research supported by Swedish Research Council grant VR 2021-03687.} \and Maryam Sharifzadeh\samethanks}

\begin{document}
\maketitle	
\begin{abstract}
Let $G$ be an $r$-partite graph such that the edge density between any two parts is at least $\alpha$. How large does $\alpha$ need to be to guarantee that $G$ contains a connected transversal, that is, a tree on $r$ vertices meeting each part in one vertex? And what if instead we want to guarantee the existence of a Hamiltonian transversal?

In this paper we initiate the study of such extremal multipartite graph problems, obtaining a number of results and providing many new constructions, conjectures and further questions.
\end{abstract}	

\section{Introduction}
\subsection{Background}
Given an $r$-partite graph $G$ with $r$-partition $\sqcup_{i=1}^r V_i$, denote its \emph{$r$-partite density} $d_r(G)$ by
\[d_r(G):=\min_{1\leq i<j\leq r} \frac{\vert E(G[V_i, V_j])\vert}{\vert V_i\vert \cdot \vert V_j\vert}.\]
(Here we assume $V_i\neq \emptyset$ for every $i$.)

A \emph{transversal} $G'$ of $G$ is a subgraph of $G$ induced by a set of vertices $U$ meeting each part $V_i$ in exactly $\vert U\cap V_i\vert = 1$ vertex. Clearly $G'$ can be viewed a subgraph of $K_r$, the complete graph on $r$-vertices. Given a family $\mathcal{F}$ of non-empty subgraphs of $K_r$, we say that $G$ has \emph{$\mathcal{F}$-free transversals} if every transversal  $G'$ of $G$ is $\mathcal{F}$-free. In this paper we are concerned with a generalisation of the following Tur\'an-type problem:
\begin{problem}[Density Tur\'an problem for multipartite graphs]\label{problem: r-partite}
Determine the supremum $\pi_r(\mathcal{F})$ of the $\alpha\leq 1$ such that there exists an $r$-partite graph $G$ with $d_r(G)=\alpha$ and $\mathcal{F}$-free transversals. 
\end{problem}
Problem~\ref{problem: r-partite} was introduced by Bondy, Shen, Thomass\'e and Thomassen~\cite{BSTT06}, who considered the case when $\mathcal{F}$ consists of a single graph, the triangle $K_3$. Bondy, Shen, Thomass\'e and Thomassen gave a beautiful proof that any tripartite  graph $G$ with $d_3(G)>\frac{-1+\sqrt{5}}{2}$ (the reciprocal of the golden ratio) must contain a triangle, and that this result is best possible. They also fully resolved a more general inhomogeneous version of the problem. Further, they showed than in any \emph{infinite}-partite graph, a partite density strictly greater than $\frac{1}{2}$ suffices to guarantee the existence of a triangle.

This latter result was improved by Pfender~\cite{Pfender12} , who showed that for $r\geq 12$, any $r$-partite graph with triangle-free transversals can have $r$-partite density at most $\frac{1}{2}$. More generally, Pfender showed that for any $t\in \mathbb{Z}_{\geq 2}$ there exists  $r\in \mathbb{Z}_{\geq t}$ such that any $r$-partite graph with $K_t$-free transversals can have $r$-partite density at most $\frac{t-2}{t-1}$, which is easily seen to be best possible by taking a suitable intersection of a $(t-1)$-partite Tur\'an graph with an $(r-1)$-partite graph. Thus, in our notation, Pfender's result is a multigraph analogue of Tur\'an's theorem stating that for every integer $t\geq 2$, $\pi_r(\{K_t\})=\frac{t-2}{t-1}$ for all $r$ sufficiently large.

Pfender's result was generalised by Narins and Tran~\cite{NarinsTran17}, who obtained a multipartite analogue of the Erd{\H o}s--Stone theorem, with a surprising twist. Given a graph $F$ with chromatic number $\chi(F)$, they showed that if $F$ is \emph{almost colour-critical} (see~\cite{NarinsTran17} for a definition of this term), then for all $r$ sufficiently large,
\[\pi_r(\{F\})=\frac{\chi(F)-2}{\chi (F)-1},\]
as one would expect; however, if $F$ is not almost-colour critical, then Narins and Tran showed that $\pi_r(\{F\})\geq \frac{\chi(F)-2}{\chi(F)-1}+\frac{1}{(\chi(F)-1)^2(r-1)^2}> \frac{\chi(F)-2}{\chi(F)-1}$ for all $r\geq \vert F\vert$.

Bondy, Shen, Thomass\'e and Thomassen's results were extended in three further directions. First of all Baber, Johnson and Talbot~\cite{BJT10} considered the problem of minimising the density of triangles in tripartite graphs with $3$-partite density above $\frac{-1+\sqrt{5}}{2}$ (which can be viewed as a multipartite version of the Rademacher--Tur\'an problem). Secondly, Markstr\"om and Thomassen~\cite{MT21} determined the partite density needed to guarantee a copy of $K_{r+1}^{(r)}$, the complete $r$-uniform hypergraph on $r+1$ vertices, in an $(r+1)$-partite $r$-uniform hypergraph.

Finally, in a direction related to the one we pursue in the present paper, Nagy~\cite{Nagy11}  studied the partite density needed in a subgraph of a blow-up of a graph $H$ to guarantee the existence of $H$ as a transversal subgraph. In particular, Nagy determined this critical density in the case when $H$ is a tree or a cycle, relating it to the spectral properties of $H$'s adjacency matrix. Nagy's work was generalised by Csikv\'ary and Nagy~\cite{CsikvariNagy12}, who obtained inhomogeneous versions and extensions of Nagy's results.

Beyond the problems considered in this paper, multipartite graphs are widely-studied objects in extremal graph theory. Multipartite graphs appear in applications of the Szemer\'edi regularity lemma, and are the subject of an influential family of problems posed by Bollob\'as, Erd{\H o}s and Szemer\'edi~\cite{BES75}, on which research is still ongoing~\cite{LTZ22}. Finally, there is a connection between multipartite graphs (and especially the more general $H$-partite graphs we shall focus on) and $1$-dependent random graph models, which we expound on in Section~\ref{section: connection to 1dep}. Given the applications of such $1$-dependent models in percolation theory~\cite{BBW09, RW07}, this gives strong motivation for studying the connectivity-related questions we shall consider. 
\subsection{Setting of this paper}
In this paper we study a slightly more general form of Problem~\ref{problem: r-partite}. Similarly to Nagy~\cite{Nagy11}, given a host graph $H$ we will consider \emph{$H$-partite graphs}, which we define below.
\begin{definition}[$H$-partite graphs]\label{def: H-graph}
	A \emph{weighted $H$-partite graph} (henceforth abbreviated to $H$-partite graph) is a graph $G$ together with a \emph{canonical $H$-partition} $V(G):=\bigsqcup_{x\in V(H)}V_x$ and a weight function $w: \ V(G)\rightarrow [0,1]$ satisfying the following:
	\begin{enumerate}
		\item for each $x\in V(H)$, $\sum_{v\in V_x} w(v)=1$ (i.e.\ for each $x\in V(H)$,  $w$ gives rise to a probability distribution over $V_x$); 
		\item $E(G)= \bigcup_{xx'\in E(H)} E(G[V_x, V_y])$ (i.e.\  $G$ is a subgraph of the blow-up of the host graph $H$ associated with the canonical partition $\bigsqcup_{x\in V(H)}V_x$).
	\end{enumerate}
\end{definition}
\begin{remark}\label{remark: equivalence weighted H-graphs and multipartite graphs}
Any subgraph of a blow-up of the host graph $H$ can be viewed as a weighted $H$-partite graph by letting the weight function correspond to the uniform distribution over each part $V_x$, $x\in V(H)$ (i.e.\ setting $w(v)=1/\vert V_x\vert $ for every $v\in V_x$). Conversely, given any weighted $H$-partite graph $G$ where the weight function $w$ takes rational values only, one can construct a subgraph $G'$ of a blow-up of $H$ in a natural way: let $N\in \mathbb{N}$ be chosen so that $Nw(v)\in \mathbb{N}$ for all $v\in V(G)$ with $w(v)>0$. Then replace each $v\in V(G)$ by a set $B_v$ of $Nw(v)$ vertices, and put a complete bipartite graph between $B_u$ and $B_v$ whenever $uv\in E(G)$. Since any weight function on a finite graph  can be arbitrarily well-approximated by a rational weight function, it follows that the study of subgraphs of blow-ups of $H$ and weighted $H$-partite graphs are (asymptotically) equivalent. We thus choose to refer to the slightly more general class of weighted $H$-partite graphs as $H$-partite graphs.
\end{remark}
\noindent The notion of $r$-partite density generalises to the $H$-partite setting in a natural way:
\begin{definition}[$H$-partite density]
	Given an $H$-partite graph $G$, we define the \emph{$H$-partite density profile} of $G$ to be
	\[\mathbf{\alpha}(G):= \Bigl(\alpha_{xy}\Bigr)_{xy \in E(H)} \ ,\]
	where $\alpha_{xy}:=\sum_{u\in V_x, v\in V_y} w(u)w(v)\mathbbm{1}_{uv\in E(G)}$ is the edge density of the (weighted) bipartite subgraph of $G$ induced by $V_x \sqcup V_y$. Further, the \emph{$H$-partite density} of $G$ is $d_H(G):=\min\left\{\alpha_{xy} :\ xy \in E(H)\right\}$.
\end{definition}
\noindent Similarly to the $r$-partite setting, a \emph{transversal} of an $H$-partite graph $G$ is a subgraph of $G$ induced by a set of vertices $U$ meeting each part $V_x$, ${x\in V(H)}$ of the canonical partition of $G$ in exactly one vertex. Our interest in this paper is then the following generalisation of Problem~\ref{problem: r-partite} 
\begin{problem}\label{problem: Turan for multipartite graphs}[Density Tur\'an problem for $H$-partite graphs]
	Given a non-empty family $\mathcal{F}$ of non-empty subgraphs of a host graph $H$, determine
	\[\pi_H(\mathcal{F}):=\sup\Bigl\{d_H(G): \ G \textrm{ is an $H$-partite graph with $\mathcal{F}$-free transversals} \Bigr\}.\]
\end{problem}
\noindent By a result of Bondy, Shen, Thomass\'e and Thomassen, for finite host graphs $H$ the supremum in the definition of $\pi_H(\mathcal{F})$ is in fact a maximum (see Proposition~\ref{prop: bondyetal}).

\subsection{Results}
We investigate the $H$-partite density threshold for connected transversals in $H$-partite graphs. Given a positive integer $r$, let $\mathcal{T}_r$ denote the family of all trees on $r$ vertices. Our first result is as follows: 
 for every graph $H$ obtained from $K_n$ by deleting a non-empty matching, we determine the smallest $H$-partite density forcing the existence of a connected transversal in an $H$-partite graph.
\begin{theorem}\label{theorem: transversal trees for Kr-matching}
Let $r\geq 4$ and let $M$ be a non-empty matching in $K_r$. Then for the graph $H=K_r-M$ obtained from $K_r$ by deleting the edges of the matching $M$, we have
	\[\pi_{H}(\mathcal{T}_r)=\frac{1}{2}.\]
\end{theorem}
\noindent We show however that when $H=K_r$, the connectivity threshold dips strictly below $1/2$:
\begin{theorem}\label{theorem: crude bound on connectivity for Kr}
For all $r\geq 4$,
	\[ \frac{r-2}{2(r-1)^2} \left(3r-4-\sqrt{5r^2-16r+12} \right)\leq \pi_{K_r}(\mathcal{T}_r)\leq \frac{1}{2}-\frac{1}{4r-6}.\]
\end{theorem}
\begin{remark}
It is easily checked that for $r\geq 4$,  the sequence given by $u_r:=\frac{r-2}{2(r-1)^2} \left(3r-4-\sqrt{5r^2-16r+12}\right)$ is strictly increasing and satisfies $u_4=\frac{8-2\sqrt{7}}{9}= 0.300944\ldots$ and $\lim_{r\rightarrow \infty}u_r=\frac{3-\sqrt{5}}{2}=0.381966\ldots$ .
\end{remark}
\noindent We conjecture that the lower bound in Theorem~\ref{theorem: crude bound on connectivity for Kr} is tight , and in particular that a $K_r$-partite density of $\frac{3-\sqrt{5}}{2}$ guarantees the existence of a connected transversal:
\begin{conjecture}\label{conj: connectivity}
	The lower bound in Theorem~\ref{theorem: crude bound on connectivity for Kr} is tight. In particular, for any $r\geq 4$, every $r$-partite graph with $r$-partite density at least $\frac{3-\sqrt{5}}{2}$ contains a transversal tree on $r$ vertices.
\end{conjecture}

\noindent In addition to these results, we prove some general upper bounds on the threshold for the existence of connected transversals.
\begin{proposition}\label{prop: general connectivity upper bound} 
	Let $H$ be any connected graph on $r$ vertices. Then $\pi_{H}(\mathcal{T}_r)\leq \frac{r-2}{r-1}$.
\end{proposition}
\noindent This elementary bound is sharp in general, as can be seen by considering stars.
\begin{proposition}\label{prop: star connectivity}
	For the star $K_{1,r-1}$, we have $\pi_{K_{1, r-1}}(\mathcal{T}_r)=\frac{r-2}{r-1}$
\end{proposition}
\noindent For $r$ even, the ladder on $r$ vertices is the Cartesian product\footnote{Recall that given two graphs $G_1$ and $G_2$, their \emph{Cartesian product} $G_1\times G_2$ is the graph on the set $V(G_1)\times V(G_2)$ in which vertices $(u_1, u_2)$ and $(v_1, v_2)$ are joined by an edge if either $u_1v_1 \in E(G_1)$ and $u_2=v_2$ or $u_1=v_1$ and $u_2v_2\in E(G_2)$.} $K_2\times P_{r/2}$  of a single edge with a path of length $r/2$. As a consequence of works of Nagy~\cite{Nagy11} and Day, Falgas--Ravry and Hancock~\cite{DFRH20}, $\pi_H(\mathcal{T}_r)$ is known when $H$ is a tree or a ladder on $r$ vertices. Since clearly $\pi_H(\mathcal{T}_r)\leq \pi_{H'}(\mathcal{T}_r)$ whenever $H'$ is a supergraph of $H$, this yields the following:
\begin{proposition}\label{prop: Ham path and ladder}
	Let $r\in \mathbb{Z}_{\geq 2}$, and let $H$ be a graph on $r$ vertices. If $H$ contains a Hamiltonian path $P_r$, then $\pi_H(\mathcal{T}_r)<\frac{3}{4}$, while if $H$ contains a spanning ladder $K_2\times P_{r/2}$, then $\pi_H(\mathcal{T}_r)< \frac{2}{3}$.
\end{proposition}

\noindent We also investigate small cases of Problem~\ref{problem: Turan for multipartite graphs}. Up to isomorphism, the connected graphs on $4$-vertices consist of $K_4$, $K_4^-$ ($K_4$ with an edge removed), $K_{2,2}=C_4$ (the cycle on $4$ vertices), $K_{1,3}$ (the star on $4$ vertices), $P_4$ (the path on $4$ vertices) and $K_4-P_3$ ($K_4$ with a path on three vertices deleted).  We summarise our bounds for $\pi_H(\mathcal{T}_4)$ for these graphs $H$ as well as for the $5$-cycle $C_5$ in the table below.

\vspace*{.1in}
\begin{tabular}{|l|l|l|}
	\hline
	Graph $H$&  Connected transversal threshold & Result\\
	\hline \hline
	$K_4$ & $\pi_{K_4}(\mathcal{T}_4)\stackrel{?}{=} \frac{8-2\sqrt{7}}{9}=0.3009\ldots $  & Theorem~\ref{theorem: crude bound on connectivity for Kr}/Conjecture~\ref{conj: connectivity} \\
	& $\pi_{K_4}(\mathcal{T}_4)< 2-2\sqrt{\frac{2}{3}}= 0.36701\ldots $ & Theorem~\ref{theorem: crude upper bound on connectitivity for 4-partite}\\
	$K_4^-$ &$\pi_{K_4^-}(\mathcal{T}_4)=\frac{1}{2}$ & Theorem~\ref{theorem: transversal trees for Kr-matching}\\
	$K_{2,2}=C_4$ & $\pi_{K_{2,2}}(\mathcal{T}_4)=\frac{1}{2}$& Theorem~\ref{theorem: transversal trees for Kr-matching}\\
	$K_4-P_3$ & $\pi_{K4-P_3}(\mathcal{T}{4})=4 -2\sqrt{3}=0.5358\ldots$& Proposition~\ref{prop: triangle with a pendant edge}\\
	$P_4$  &  $\pi_{P_4}(\mathcal{T}_{4})=\frac{-1+\sqrt{5}}{2}=0.6180\ldots$ & Nagy~\cite[Corollary 3.13]{Nagy11}\\
	$K_{1,3}$ & $\pi_{K_{1,3}}(\mathcal{T}_4)=\frac{2}{3}$& Proposition~\ref{prop: star connectivity}\\
	$C_5$ & $\pi_{C_5}(\mathcal{T}_5)=\frac{1}{2}$& Proposition~\ref{prop: C5}\\
	\hline 
\end{tabular}
\vspace*{.1in}

\noindent Finally, we consider multipartite versions of Dirac's theorem. Let $\mathcal{C}_r$ denote the collection of Hamilton cycles on $r$ labelled vertices. We prove (perhaps surprisingly) that the $r$-partite density threshold for Hamiltonicity in $r$-partite graphs is strictly greater than $1/2$.
\begin{theorem}\label{theorem: Dirac lower bound}
Fix $r\in \mathbb{Z}_{\geq 4}$. Let $p_{\star}=p_{\star}(r)$ be the unique real solution in $(\frac{1}{2}, 1)$ to the cubic equation
\begin{align}\label{eq: cubic}
(r-2)-(4r-10)p+(6r-14)p^2-(4r-8)p^3=0.
\end{align}
Then $\pi_{K_r}(\mathcal{C}_r)\geq {(p_{\star})}^2+(1-p_{\star})^2>\frac{1}{2}$.
\end{theorem}
\begin{remark}
	Asymptotic analysis of~\eqref{eq: cubic} yields that $p_{\star}=\frac{1}{2}+\frac{r^{-1}}{2}+O\left(r^{-2}\right)$ and hence that ${p_{\star}}^2+(1-p_{\star})^2=\frac{1}{2}+\frac{1}{2r^2}+O\left(\frac{1}{r^3}\right)$. 
\end{remark}
\noindent We conjecture, however, that the Hamiltonicity threshold should converge to $1/2$ as $r\rightarrow \infty$.
\begin{conjecture}[Multipartite Dirac Conjecture]\label{conj: Dirac}
$\lim_{r\rightarrow \infty} \pi_{K_r}(C_r)  =\frac{1}{2}$. \end{conjecture}
\noindent Here again, we prove some lower bounds for the $r=4$ case.
\begin{proposition}\label{prop: C4 bound}
	\begin{align*}
	%
%
	0.5707\ldots 	&\leq \pi_{K_4}(\mathcal{C}_4)\leq \frac{1}{\sqrt{3}}=0.5773\ldots.\end{align*}
\end{proposition}
\noindent We also note that results of Pfender give a tight bound on the threshold for  the emergence of odd cycles in transversals.
\begin{proposition}\label{prop: odd cycles}
For all $r\geq 4$, $\pi_{K_r}\left(\bigcup\{C_t: \ t \textrm{ odd}\}\right)\geq\frac{1}{2}$, with equality for all $r\geq 12$.	
\end{proposition}

\subsection{Organisation of the paper}
We gather some preliminary remarks in Section~\ref{section: preliminaries}. Our main results on connected transversals, Theorems~\ref{theorem: transversal trees for Kr-matching} and~\ref{theorem: crude bound on connectivity for Kr} are proved in Section~\ref{section: transversal trees}, while our result on Hamiltonian transversals can be found in Section~\ref{section: Dirac}.  In Section~\ref{section: small cases} we determine the connectivity threshold for $K_4-P_3$ and $C_5$. Finally we end the paper in Section~\ref{section: conclusion} with some remarks and discussion regarding a long list of related problems.

\subsection{Notions and notation}
We use $[r]$ to denote the set $\{1,2,\ldots, r\}$ and, given a set $S$, we write $S^{(t)}$  for the collection of all subsets of $S$ of size $t$. Throughout the paper, we use standard graph theoretic notions and notation, which we recall below for the reader's convenience.

 A graph is a pair $G=(V,E)$, where $V=V(G)$ is a set of vertices and $E=E(G)\subseteq V^{(2)}$ forms the edges of the graph. A subgraph of $G$ is a graph $G'$ with $V(G')\subseteq V(G')$ and $E(G')\subseteq E(G)$. A spanning subgraph of $G$ is a subgraph $G'$ with $V(G')=V(G)$. Given a set $U\subseteq V(G)$, the subgraph of $G$ induced by $U$ is the graph $G[U]:=(U, E(G)\cap U^{(2)})$.

The neighbourhood $N_G(x)$ of a vertex $x$ in a graph $G$ is the collection of vertices $y\in V(G)$ such that $\{x,y\}$ is an edge of $G$; the degree $\mathrm{deg}(x)=\mathrm{deg}_G(x)$ of $x$ is the size of its neighbourhood. The adjacency matrix of a graph $G$, $A=A(G)$ is a matrix with rows and columns indexed by vertices of $V$, with $A_{u,v}=\mathbbm{1}_{\{u,v\}\in E(G)}$.

A path of length $\ell-1$ in a graph $G$ is a sequence of $\ell$ distinct vertices $\{v_1, v_2, \ldots v_{\ell}\}$ with $\{v_i,v_{i+1}\}\in E(G)$ for all $i\in [\ell-1]$. Two vertices in a graph are connected if they are joined by a finite path; this is an equivalence relation on $V(G)$, whose equivalence classes are the connected components of $G$. A graph is connected if it consists of a single connected component. A tree is a minimally connected graph; a vertex of degree $1$ in a tree is called a leaf.

We denote by $K_r$ the complete graph on $r$ vertices $K_r=([r], [r]^{(2)})$, and $K_{r,s}$ the complete bipartite graph whose vertex-set is the disjoint union of an $r$-set $A$ and an $s$-set $B$, and whose edges include all pairs $\{a,b\}$ with $a\in A,b\in B$. The graph $K_{1,r-1}$ is known as the star on $r$ vertices. We also denote by $C_r$ the cycle on $r$ vertices, $C_r=\left([r], \{\{i,i+1\}: \ i\in [r-1]\}\cup \{\{r,1\}\}\right)$. A matching in a graph $G$ is a collection of vertex-disjoint edges.

Whenever there is no ambiguity, we write $uv$ for a pair $\{u,v\}$. Similarly we write $F$ for the subgraph family $\{F\}$. Finally, when we consider $H$-partite graphs with $H=K_r$, the complete graph on $r$-vertices, we write `$r$-partite'  rather than `$K_r$-partite', and similarly `$d_r(G)$' and `$\pi_r(\mathcal{F})$\ rather than `$d_{K_r}(G)$' and `$\pi_{K_r}(\mathcal{F})$', so as to avoid notational clutter. In the context of (weighted) $H$-partite graphs, given a set $A$ of vertices we write $w(A):=\sum_{a\in A} w(a)$ for the sum of the weights of the vertices contained therein.

\section{Preliminaries}\label{section: preliminaries}

\subsection{Compactness and computability}
Using a simple convexity argument, Bondy, Shen, Thomass\'e and Thomassen~\cite{BSTT06} proved the useful fact that for the problems we are considering, there always exists a finite extremal example in which all the parts have bounded sizes. Explicitly, they showed\footnote{Formally, Bondy, Shen, Thomass\'e and Thomassen contented themselves with making this observation when $H=K_3$, but their argument easily generalises to the statement of Proposition~\ref{prop: bondyetal}, as noted by several authors, see e.g.\ Nagy~\cite[Lemma 2.1]{Nagy11} for a formal proof.}:
\begin{proposition}[Bondy, Shen, Thomass\'e, Thomassen]\label{prop: bondyetal}
For any non-empty host graph $H$ and any family  $\mathcal{F}$ of non-empty subgraphs of $H$, there exists an $H$-partite graph $G$ such that \begin{enumerate}[(i)]
	\item (extremality) $G$ has $\mathcal{F}$-free transversals and $H$-partite density $\pi_H(\mathcal{F})$;
	\item (finiteness) for every vertex $x\in V(H)$, the corresponding part $V_x$ in the canonical partition of $G$ contains at most $\mathrm{deg}_H(x)$ vertices.
\end{enumerate}	
\end{proposition}
\noindent So in principle, determining $\pi_H(\mathcal{F})$ is a finitely computable problem --- albeit one in which the complexity rises very quickly as the average degree in $H$ increases.

\subsection{Random transversals and connection to $1$-dependent random graphs}\label{section: connection to 1dep}
A useful tool when studying  $H$-partite graphs is to consider random transversals. More specifically, given an $H$-partite graph $G$, a  \emph{random transversal} of $G$ is obtained by first selecting a representatives $\mathbf{v}_x \in V_x$ independently at random from each of the parts $\left(V_x\right)_{x\in V(H)}$ of the canonical $H$-partition of $G$, with $\mathbb{P}\left(\mathbf{v}_x=u\right)=w(u)$ for every $x\in V(H)$ and $u\in V_x$. The random transversal $T$ is then the subgraph of $G$ induced by the randomly chosen representatives, $T:=G\left[\{\mathbf{v}_x: \ x\in V(H)\}\right]$ (which may be viewed as a random spanning subgraph of $H$ in the natural way).

Random transversals are a special class of $1$-dependent random graphs, whose definition we recall below. Given a host graph $H=(V(H), E(H))$ and a probability measure $\mu$ on the collection of subsets of $E(H)$, we can build a random graph $\mathbf{H}$ from $H$ by setting  $V(\mathbf{H})=V(H)$ and letting $E(\mathbf{H})$ be a $\mu$-random subset of $E(H)$, i.e.\ a subset of $E(H)$ chosen at random according to the probability distribution given by $\mu$. Thus $\mathbf{H}$ is a random variable taking values in the collection of spanning subgraphs of $H$, and we refer to it as a \emph{random graph model on $H$}. 
\begin{definition}[$1$-dependent random graph models]
	Let $H$ be a host graph. A  random graph model $\mathbf{H}$ on $H$ is said to be \emph{$1$-dependent}  if whenever $A$ and $B$ are disjoint subsets of $V(H)$, the random induced subgraphs $\mathbf{H}[A]$ and $\mathbf{H}[B]$ are mutually independent random variables.
\end{definition}
\noindent Informally, a random graph model is $1$-dependent if events `living' (defined by what happens) on disjoint vertex-subsets are mutually independent. Denote by $\mathcal{M}_{1,\geq p}(H)$ the collection of $1$-dependent random graph models $\mathbf{H}$ on $H$ such that for each edge $xy\in E(H)$, $\mathbb{P}(xy \in E(\mathbf{H}))\geq p$.

It is clear that given an $H$-partite graph $G$ with $d_H(G)\geq p$, taking a random transversal of $G$ gives us a $1$-dependent random graph model over $H$ from $\mathcal{M}_{1,\geq p}(H)$. While in general there exist $1$-dependent random graphs that cannot be constructed from $H$-partite graphs, $H$-partite graphs are arguably the most important class of $1$-dependent random graph models from the point of view of applications and constructions.

As noted by Balister and Bollob\'as~\cite{BB12}, ``$1$-dependent percolation models have become a key tool in establishing bounds on critical probabilities'' in percolation theory. Despite their usefulness, however, many basic questions about the behaviour of $1$-dependent models are open.

With regards to the problems studied in this paper, set
\[p_{\textrm{conn}}(H):=\sup \{p\in [0,1]: \ \exists \mathbf{H}\in \mathcal{M}_{1,\geq p}(H) \textrm{ such that }\mathbb{P}\left(\mathbf{H} \textrm{ is connected}\right)=0 \}\]
be the $1$-dependent critical probability for connectivity over $H$. By the observation above that $H$-partite graphs correspond to a  special class of $1$-dependent models, it follows that for any connected host graph $H$ on $r$ vertices,
\begin{align}\label{eq: 1dep up on conn for H}
\pi_{H}(\mathcal{T}_r)\leq p_{\textrm{conn}}(H).\end{align}
\noindent Day, Falgas--Ravry and Hancock~\cite[Theorems 15, 16]{DFRH20} showed that $p_{\textrm{conn}}(P_r)=\frac{1}{4}\left(3-\tan^2\left(\frac{\pi}{2r}\right)\right)$, which is exactly the value of $\pi_{P_r}(P_r)$ determined by Nagy~\cite{Nagy11} (and implies his result), and that $p_{\textrm{conn}}(K_r)=\frac{1}{2}\left(1-\tan^2\left(\frac{\pi}{2r}\right)\right)$ (in which case we do not believe we have equality in~\eqref{eq: 1dep up on conn for H}).

Further in~\cite[Theorem 26]{DFRH20}, Day, Falgas--Ravry and Hancock showed that $p_{\textrm{conn}}(P_r\times K_2)<\frac{2}{3}$ for every $r\geq 1$; here $P_r\times K_2$ denotes the ladder graph on $2r$ vertices (the Cartesian product of the path $P_r$ with the edge $K_2$, obtained by taking two disjoint copies of $P_r$ and joining each vertex in one copy by an edge to the corresponding vertex in the other copy). Proposition~\ref{prop: Ham path and ladder} follows as an immediate corollary of our discussion so far.
\begin{proof}[Proof of Proposition~\ref{prop: Ham path and ladder}]
Let $H$ be an $r$-vertex graph. If $H$ contains a Hamiltonian path, then by monotonicity, comparison with $1$-dependent random graph models on $H$ and~\cite[Theorems 15]{DFRH20}, $\pi_H(\mathcal{T}_r)\leq \pi_{P_r}(P_r)\leq p_{\textrm{conn}}(P_r)<\frac{3}{4}$. Similarly if $r$ is even and $H$ contains a copy of the ladder on $r$ vertices, then it follows from~\
\cite[Theorem 26]{DFRH20} (applied with $\alpha=\frac{1}{2}$ and $f_{1, K_2}(p)=p$) that $\pi_H(\mathcal{T}_r)<\frac{2}{3}$.	
\end{proof} 
Finally, we note the proofs of~\cite[Theorems 30 and 26]{DFRH20} implicitly determine $\pi_{P_r\times K_2}(\mathcal{T}_{2r})=p_{\mathrm{conn}}(P_r\times K_2)$, though the common value of these two quantities is only given as the optimal solution to a (complicated) set of recursive inequalities.
\subsection{Paths, stars and trees}\label{subsection: paths, stars and trees}
As mentioned in the introduction, Nagy~\cite[Theorem 3.9]{Nagy11} determined $\pi_T(T)$ for every tree $T$ on $r$ vertices in terms of the largest eigenvalue of $T$'s adjacency matrix $A(T)$.
\begin{theorem}[Nagy]\label{theorem: Nagy}
	Let $T$ be a tree on $r$ vertices, and let $\lambda$ be the largest eigenvalue of $A(T)$. Then $\pi_T(T)=1-\frac{1}{\lambda^2}$.
\end{theorem}
\noindent In particular, using results of Lov\'asz and Pelik\'an, Nagy's result implies that for the path $P_r$ on $r$ vertices we have
\begin{align}\label{eq: paths}
\pi_{P_r}(P_r)=1-\frac{1}{4\cos^2(\frac{\pi}{r+1})};
\end{align}
see~\cite[Corollary 3.13]{Nagy11} or~\cite[Theorem 16]{DFRH20} for an elementary proof of a stronger result using a variant of the Lov\'asz local lemma.

For the sake of making this paper self-contained, we give here short proofs of the weaker Propositions~\ref{prop: general connectivity upper bound} and~\ref{prop: star connectivity} as well as of an illustrative `star absorption' lemma.
\begin{lemma}[Star absorption lemma]\label{lemma: star absorption}
	Consider a $K_{1,N}$-partite graph $G$ with partite density $p>1/2$. Suppose that there is a subset $A$ of the part corresponding to the centre of the star $K_{1, N}$ such that $w(A)=\alpha>1-p$. Then there exists a vertex $a\in A$ such that all but at most $\Bigl(\frac{1-p}{\alpha}\Bigr)N$ of the leaf-parts of $G$ contain a vertex joined to $a$ by an edge of $G$.
\end{lemma}
\begin{proof}
	Suppose no vertex $a\in A$ sends an edge to more than $\theta N$ of the leaf-parts of $G$. Then by the partite density condition, we have:
	\begin{align*} 
	pN&\leq (1-\alpha)N + \theta \alpha N.
	\end{align*}	
	Rearranging yields $1-\theta \leq\frac{1-p}{\alpha}$. In particular, there exists a vertex $a\in A$ such that all but at most  a $\frac{1-p}{\alpha}$ proportion of the leaf-parts of $G$ send an edge to $a$.
\end{proof}
Let us now sketch how this lemma may be applied to find transversals with large or connected components. Consider an arbitrary part $V_0$ in an $(r+1)$-partite graph $G$ with $(r+1)$-partite density $p\ge 1/2$. By averaging and relabelling, there exists a vertex $v_0$ in $V_0$ sending edges to an $\alpha_i$ proportion of part $V_i$ (meaning a subset of $V_i$ with weight $\alpha_i$) for $1\leq i \leq r$, where $\alpha_1\geq \alpha_2 \geq \alpha_3 \geq \ldots \geq \alpha_r$  and $\sum_{i=1}^{r}\alpha_i\geq pr$.

Observe that if $\alpha_{r/2}> 1/2$ then  by the pigeon-hole principle and our partite density condition, for each $i \in \{1,2, \ldots, r/2\}$ there exists a vertex in $V_{r/2+i}$ sending an edge pf $G$ to $N(v_0)\cap V_i$. We can thus find a transversal copy  of a $2$-subdivision of a star with $r/2$ leaves and $v_0$ as its centre, and hence a connected transversal of $G$. Also if $\alpha_r>0$, then $G$ contains a transversal copy of $K_{1,r}$, i.e.\ a connected transversal.

On the other hand if $\alpha_{r/2}\leq 1/2$ and $\alpha_r=0$, then there exists a least $t<r/2$ such that $\alpha_{t+1}\leq 1/2$, and a greatest $s\geq 1$ such that $\alpha_{r+1-s}=0$.  The $\alpha_i$ then satisfy
\begin{align}\label{eq: density inequality}
pr\leq  \frac{r-t-s}{2} +\sum_{i=1}^{t}\alpha_i.
\end{align}
Applying Lemma~\ref{lemma: star absorption} successively to $V_1$, $V_2, \ldots, V_t$ to `absorb' the bad parts $V_{r+1-s}, \ldots, V_{r}$, we see that provided 
\begin{align}\label{eq: absorption condition}
s\prod_{i=1}^t\left(\frac{1-p}{\alpha_i}\right) <1,	
\end{align}	
there exists a transversal connected subgraph of $G$ (consisting of a central vertex joined by an edge to the centres of a `star-forest'  (a collection of vertex-disjoint stars on $r$ vertices). This simple idea underlies the proofs of Theorems~\ref{theorem: transversal trees for Kr-matching} and~\ref{theorem: crude bound on connectivity for Kr} and may have further applications to the study of $H$-partite graphs.

\begin{proof}[Proof of Proposition~\ref{prop: general connectivity upper bound}]
	Recall that $H$ is an $r$-vertex graph. Let $G$ be an $H$-partite graph with $H$-partite density $p>(r-2)/(r-1)$. Let $\mathbf{T}$ be a random transversal of $G$, and let $S$ be any spanning tree of $H$. Then 
	\[\mathbb{E}\vert E(\mathbf{T})\cap E(S)\vert = p(r-1)>r-2.\]
	It follows that with strictly positive probability $\vert E(\mathbf{T})\cap E(S)\vert =r-1$, and thus that $H$ contains a connected transversal.
\end{proof}

\begin{proof}[Proof of Proposition~\ref{prop: star connectivity}]
	The upper bound was proved in Proposition~\ref{prop: general connectivity upper bound}, and also follows directly from an application of Lemma~\ref{lemma: star absorption} with $\alpha=1$, $N=r-1$ and $1-p<\frac{1}{r-1}$. For the lower bound, consider a $K_{1,r-1}$-partite graph obtained by setting the centre part $V_0$ of the star to be $[r-1]$ with the uniform weighting $w(i)=1/(r-1)$, letting each of the $r-1$ leaf parts $V_i$, $1\leq i\leq r-1$ consist of a single vertex $v_i$, and adding all edges $iv_j$, $1\leq i ,j \leq r-1$ with $i\neq j$. This is easily seen to have no spanning connected transversal, and $K_{1,r}$-partite density $\frac{r-2}{r-1}$. 
\end{proof}



\section{Transversal trees in $H$-partite graphs}\label{section: transversal trees}

\subsection{$K_r$-partite graphs}

\begin{construction}\label{construction: leila}
Fix $\alpha\in [0,1]$. For $r \geq 3$, we construct a weighted $r$-partite graph $G^b=G^b(\alpha)$ as follows: for $i\in [r-1]$, we set $V_i:=\{i,r\}$ with weight function $w(i)=\alpha$   and $w(r)=1-\alpha$, while we set $V_r:=[r-1]$ with the uniform weight function $w(i)=1/(r-1)$ for all $i\in [r-1]$. A partite edge $uv$ is present in $G^b$ if and only if $u=v=r$ or $u\in V_r$ and $u\neq v$.
\end{construction}
\begin{proposition}\label{prop: Kr connectivity lower bound}
For $\alpha=\frac{1}{2(r-1)} \left(3r-4-\sqrt{5r^2-16r+12}\right)$,
 the graph $G^b$ has $r$-partite density 
 \[\rho^b(r):=\frac{r-2}{2(r-1)^2} \left(3r-4-\sqrt{5r^2-16r+12}\right),\]
 and contains no connected transversal.  
\end{proposition}
\begin{proof}
Consider any transversal $T$ of $G^b$, with $t_i$ the vertex in $T\cap V_i$.  Then $t_r=j$ for some $j\in [r-1]$, and by construction of $G^b$ the transversal $T$ induces one of the following: the disjoint union of a star centred at $t_r$ and a non-empty clique containing $t_j$, or the disjoint union of a star centred at $t_r$, a (possibly empty) clique and an isolated vertex $t_j$. In either case, the spanning transversal is not connected, establishing half of the proposition.

For the other half, observe that the edge density between parts $V_i$ and $V_j$ for $1\leq i <j\leq r$ is $(1-\alpha)^2$ if $j<r$ and $\frac{r-2}{r-1}\alpha $ if $j=r$. The function $\min\left(\frac{r-2}{r-1}\alpha , (1-\alpha)^2\right)$ is maximised for $\alpha\in [0,1]$ when
\begin{align*}
\alpha^2 -\frac{3r-4}{r-1}\alpha +1=0
\end{align*}
i.e . \ when $\alpha=\frac{1}{2(r-1)} \left(3r-4-\sqrt{5r^2-16r+12}\right)$, at which point it attains the value $\rho^b(r)$.
\end{proof}

\begin{proof}[Proof of Theorem~\ref{theorem: crude bound on connectivity for Kr}]
The lower bound on $\pi_{K_r}(\mathcal{T}_r)$ follows from Proposition~\ref{prop: Kr connectivity lower bound}. For the upper bound, consider an $r$-partite graph $G$ with edge density $p>\frac{r-2}{2r-3}$. Suppose for contradiction that $G$ contains no connected transversal. By averaging, there exists a vertex $v_r\in V_r$ such that 
\begin{align}\label{eq: average density bound}
	\sum_{i\in [r-1]} w(N_G(v_r)\cap V_i)\geq (r-1)p.
	\end{align} 
	We may assume without loss of generality that $\alpha_i:=w(N(v_r)\cap V_i)$ is a decreasing sequence. If $\alpha_{r-1}>0$, then for every $i\in [r-1]$ there exists $v_i \in V_i$ with $v_iv_r\in E(G)$, and $G$ contains a transversal star on $r$ vertices, a contradiction. Let therefore $s>0$ be the maximal integer such that $\alpha_{r-s+1}=0$.  By averaging, we have $\alpha_1\geq \frac{r-1}{r-1-s}p>\frac{r-1}{2r-3}$. Let $t\geq 1$ be the maximal integer such that $\alpha_t\geq \frac{r-1}{2r-3}$.

 Observe that for any $i\leq t$ and $j>r-s$, there exist vertices $v_i\in V_i$ and $v_j\in V_j$ with $v_iv_j\in E(G)$. Indeed, by our choice of $p$, we have  $\alpha_i\geq\frac{r-1}{2r-3}> 1-p$. Thus if $t\geq s$, we have that $G$ contains a subdivision of a star on a total of $r$ vertices as a transversal, a contradiction. We may therefore assume $s>t$.

 Now by~\eqref{eq: average density bound} we have 
 \begin{align*}
\frac{(r-1)(r-2)}{2r-3}< (r-1)p\leq\sum_{i=1}^{r-1}\alpha_i\leq  (r-1-s-t)\frac{r-1}{2r-3} +t,
 \end{align*}
 which after rearranging terms and using our assumption $s>t$ implies 
 \begin{align*}
 \frac{r-1}{2r-3}> \frac{s(r-1)-t(r-2)}{2r-3}\geq \frac{(r-1)+t}{2r-3},
 \end{align*}
 a contradiction.
\end{proof}

For the special case $r=4$, we can prove a slightly better upper bound on the threshold for connected transversals.
\begin{theorem}\label{theorem: crude upper bound on connectitivity for 4-partite}
$\pi_4(\mathcal{T}_4)\leq 2-2\sqrt{\frac{2}{3}}= 0.36701\ldots $.
\end{theorem}
\begin{proof}
Consider a $4$-partite graph $G$ with edge density $p:1/3 \leq p<2/3$ and no transversal tree on $4$ vertices. Then each vertex of $G$ sends edges to at most two parts of the canonical partition of $G$. By averaging and relabelling parts if necessary, it follows that there exists a vertex $v_1\in V_1$ whose neighbourhood has maximum weight over all vertices of $G$, sending no edge to $V_4$, and with $\alpha:=w(N(v_1)\cap V_2)$, $\beta:=w(N(v_1)\cap V_3)$ satisfying $\alpha\geq \beta$ and
\begin{align}\label{eq: k4crude 1}
\alpha+\beta \geq 3p. 
\end{align}
Since $p>1/3$, this implies that $\alpha$ and $\beta$ are both non-zero. Given a vertex $u\in V_4$, set
\[ (x_u, y_u, z_u):= \left(w(N(u)\cap V_1), w(N(u)\cap V_2), w(N(u)\cap V_3)\right).\]	
Since $G$ contains no transversal tree on $4$ vertices, $(x_u, y_u, z_u)$ has at most two non-zero coordinates, and sends no edge to $v$ or $N(v)$. Thus one of the following holds:
\begin{enumerate}[(1)]
	\item $x_u=0$, $y_u\leq 1-\alpha$, $z_u\leq 1-\beta$;
	\item $0<x_u\leq 1$ and $y_u=0$, and $z_u \leq \min (1-\beta, \alpha+\beta-x_u)$;
	\item $0<x_u\leq 1$, $0<y_u\leq \min(1-\alpha, \alpha+\beta-x_u)$ and $z_u=0$.
\end{enumerate}
Let $\theta_i:=w\{u\in V_4: (x_u, y_u, z_u) \textrm{ satisfies }(i)\}$. Then by the density condition on $G$ between the parts $V_4$ and $V_i$, $i\in [3]$, we have:
\begin{align}
(1-\theta_1)&\geq p \label{eq: k4crude2}\\
(1-\alpha)(1-\theta_2)&\geq p  \label{eq: k4crude3}\\
(1-\beta)(\theta_1+\theta_2)&\geq p \label{eq: k4crude4}.
\end{align}
Now adding $(\theta_1+\theta_2)$ times \eqref{eq: k4crude3} to $(1-\theta_2)$ times \eqref{eq: k4crude4} and combining it with our lower bound~\eqref{eq: k4crude 1} on $\alpha+\beta$, we obtain
\begin{align*}
(1+\theta_1)p\leq  (\theta_1+\theta_2)(1-\theta_2)(2-\alpha-\beta) \leq (\theta_1+\theta_2)(1-\theta_2)(2-3p). 
\end{align*}
Rearranging terms, we get
\begin{align}\label{eq: k4crude5}
\frac{p}{2-3p}\leq \frac{(\theta_1+\theta_2)(1-\theta_2)} {1+\theta_1}:=f(\theta_1, \theta_2).
\end{align}
Now the partial derivative of $f(x,y)$ with respect to $x$ is $(1-y)^2/(1+x)^2$, which is strictly positive for $x\geq 0$. In particular, since $\theta_1\leq 1-p$ by~\eqref{eq: k4crude2}, it follows that the right hand-side of~\eqref{eq: k4crude5} is at most 
\[f(1-p, \theta_2)= 1-\theta_2 - \frac{(1-\theta_2)^2}{2-p}.\]
Substituting this into~\eqref{eq: k4crude5} and rearranging terms, we get the following quadratic inequality for $\theta_2$:
\begin{align*}
 (\theta_2)^2-p\theta_2+\frac{7p-2-4p^2}{2-3p}\leq 0.\end{align*}
In particular, the discriminant
\[p^2-4\left(\frac{7p-2-4p^2}{2-3p}\right)=\frac{2-p}{2-3p}\left(3p^2-12p+4\right)\]
must be non-negative. Solving $3p^2-12p+4\geq 0$ for $p\leq 1$ yields $p\leq 2-2\sqrt{\frac{2}{3}}$, concluding the proof of the proposition.
\end{proof}



\subsection{$(K_r-M)$-partite graphs}\label{subsection: Kr - matching}
\begin{construction}\label{construction: kr-partite}
	Let $12$ be the missing edge in $K_r^-$. We construct a $K_r^-$-partite graph $G^b$ as follows: let $V_1=\{1\}$ and $V_{2}=\{2\}$; for every $i\in [r]\setminus [2]$, set $V_i=\{1,2\}$. Place a uniform weight function on each of the parts, and include an edge $xy\in V_i\times V_j$ (with $i\neq j$) in $G^b$ if and only if $x=y$. 
\end{construction}
\begin{proof}[Proof of Theorem~\ref{theorem: transversal trees for Kr-matching}]
	For the lower bound, it is easily checked that the $K_r^-$ partite graph $G^b$ given in Construction~\ref{construction: kr-partite} has $K_r^-$partite density exactly equal to $1/2$ and contains no connected transversal, since every transversal component can meet at most one of $V_1$ and $V_2$.

	For the upper bound, let $H$ be a graph on $[r]$ in which each vertex has degree at least $r-2$, i.e.\ a graph obtained from $K_r$ by deleting a matching. Assume without loss of generality that $[r-2]\subseteq N_H(r)$.

Consider an $H$-partite graph $G$ with $H$-partite density $p>1/2$. Let $V_1, V_2, \ldots, V_{r}$ denote the parts from the $H$-partition of $V(G)$.  By averaging and our assumption on $N_H(r)$, there exists a vertex $v_r\in V_r$ such that for every $i\in [r-2]$, there are $\alpha_i\vert V_i\vert$ edges of $G$ from $v_r$ to $V_i$, where the $\alpha_i$ are reals from $[0,1]$ satisfying:
	\begin{align}\label{eq: density weak ub}
	\frac{r-2}{2}< p(r-2)	\leq \sum_{i=1}^{r-1}\alpha_i.\end{align}	
	Let $t$ be the number of $i$ for which $\alpha_i\geq 1/2$, and let $s$ be the number of $i$ for which $\alpha_i=0$. Then
	\begin{align*}
	\sum_{i=1}^{r-2}\alpha_i\leq \frac{r-2-s-t}{2} +t=\frac{r-2+t-s}{2}.
	\end{align*}
	Combining this with~\eqref{eq: density weak ub}, we have that $t\geq s+1$. Since every vertex of $H$ is incident with at most one non-edge, it follows from this inequality that there exists an injection $f$ from the $s$-set $B:=\{i\in [r-2]:\ \alpha_i=0\}\cup \{r-1\}$ into the $t$-set $A:=\{i\in [r-2]: \ \alpha_i>1/2\}$ such that $if(i)\in E(H)$ for all $\in B$. We shall use this to find a subdivision of a star as a subgraph of a transversal of $G$.

	Pick for each $i\in [r-1]$ a vertex $v_i\in V_i$ such that (i) if $i\notin B$ then $v_iv_r\in E(G)$, while (ii) if $i\in B$ then $v_iv_{f(i)}\in E(G)$. Clearly if we can do this then we have found our desired subdivision of a star centred at $v_r$ as a subgraph of a transversal of $G$. Let us therefore verify we can find good choices of the vertices $v_i\in V_i$, $i\in [r-1]$. For $i\in [r-1]\setminus \left(B\cup f(B)\right)$, we just pick $v_i$ to  be an arbitrary vertex in $N_G(v_e)\cap V_i$ and (i) is trivially satisfied.

	Next consider $i\in B$, and the associated index $f(i)\in A$. By the edge density condition, at least a $p> 1/2$ proportion of the edges between $V_i$ and $V_{f(i)}$ are present in $G$ (since $if(i)\in E(H)$ by construction of the injection $f$). In particular, since $\vert N_G(v_r)\cap V_{f(i)}\vert=\alpha_{f(i)}\vert V_{f(i)}\vert \geq \frac{1}{2}\vert V_{f(i)}\vert$, there exist $v_i \in V_i$ and $v_{f(i)}\in  N_G(v_r)\cap V_{f(i)}$ such that $v_iv_{f(i)}$ is an edge of $G$. Clearly for these choices of $v_i$ and $v_{f(i)}$ we have that (ii) is satisfied for $i$. Further, by construction, $v_rv_{f(i)}\in E(G)$, and (i) is satisfied for the index $f(i)\notin B$. Thus there exists good choices of the vertices $v_i\in V_i$, $i\in [r-1]$, and we are done.
\end{proof}

\section{A multipartite Dirac density problem}\label{section: Dirac}
\subsection{Multipartite graphs with no Hamiltonian transversals}
We begin by giving two easy constructions that give the  lower bound of $\pi_r(C_r)\geq 1/2$ for the multipartite Dirac and odd cycle problems.
\begin{construction}[Two colour construction]\label{construction: dead-end}
	For $r\geq 4$, we construct a weighted $r$-partite graph $G_1$ as follows. We let the parts $V_i$, $i\in [r-2]$ consist of $r-2$ disjoint copies of $\{0, 1\}$. Further we let $V_{r-1}=\{0\}$ and $V_{r}=\{1\}$. We let the weight function $w$ be constant and equal to $1/2$ on $\sqcup_{i\in [r-2]}V_i$, and constant and equal to $1$ on $V_{r-1}\cup V_r$. 
	Given $x\in V_i$ and $y\in V_j$ with $i\neq j$, we let $xy$ be an edge of $G_1$ if either $x=y$ or $\{i,j\}=\{r-1,r\}$.
\end{construction}
\noindent It is easily checked that Construction~\ref{construction: dead-end} has $r$-partite density $1/2$ and contains no transversal Hamilton cycle, so that $\pi_r(C_r)\geq \frac{1}{2}$ for all $r\geq 3$. 
\begin{construction}[Parity construction]\label{construction: parity}
For $r\geq 5$, we construct a weighted $r$-partite graph $G_2$ as follows. We let the parts $V_i$, $i\in [r]$ consist of $r$ disjoint copies of $\{0, 1\}$, and the weight function to be the constant function $w: \ V(G_2)\rightarrow \{1/2\}$.  Given $x\in V_i$ and $y\in V_j$ with $1\leq i<j\leq r$, we let $xy$ be an edge of $G_2$ if and only if $x+y$ is odd. 
\end{construction}
\noindent It is easily checked that Construction~\ref{construction: parity} has $r$-partite density $1/2$ and contains no transversal cycle of odd length. In particular, for $r\geq 3$ and odd this gives an alternative proof that $\pi_r(C_r)\geq\frac{1}{2}$.
\begin{proof}[Proof of Proposition~\ref{prop: odd cycles}]
That $\pi_r\left(\bigcup\{C_t: \ t \textrm{ odd}\} \right)\geq \frac{1}{2}$ follows from Construction~\ref{construction: parity}; that we have equality for $r\geq 12$ follows from Pfender's proof~\cite{Pfender12} that $\pi_{r}(K_3)=\frac{1}{2}$ for all $r\geq 12$.
\end{proof}
\noindent We now give a refined version of Construction~\ref{construction: dead-end}, which gives an improved upper bound for $\pi_r(C_r)$.
 \begin{construction}\label{construction: refined dead-end}
 	Let $p_1, p_2, p_3\in (0,1)$. For $r\geq 4$, we construct a weighted $r$-partite graph $G_3:=G_3(p_1,p_2,p_3)$ as follows. We let the parts $V_i$, $i\in [r-2]$ consist of $r-2$ disjoint copies of $\{0, r-1\}$.  Further we let $V_{r-1}=\{0, r\}$ and $V_{r}=\{1,2, \ldots r-2, r-1\}$, with all parts pairwise disjoint. 
 	
 	We now define a weight function $w$ in the following way. If $v\in V_i$ for some $i\in [r-2]$, then $w(v)=p_1$ if $v=0$ and  $w(v)=1-p_1$ otherwise. If $v\in V_{r-1}$, then $w(v)=p_2$ if $v=0$ and $w(v)=1-p_2$ otherwise. If $v\in V_r$, then $w(v)=p_3$ if $v=r-1$, and $w(v)=(1-p_3)/(r-2)$ otherwise.

Finally, we specify the edges of $G$. Given $x\in V_i$ and $y\in V_j$ with $1\leq i<j\leq r-1$, we let $xy$ be an edge of $G_3$ if and only if $x=y=0$ or $x=y=r-1$.  For $x\in V_i$ with $i\leq r-2$ and $y\in V_r$, we let $xy$ be an edge of $G_3$ if either $x=r-1$ or $y=i$. Finally if $x\in V_{r-1}$ and $y\in V_r$, we let $xy$ be an edge if either  $x=0$ and $y=r-1$ or $x=r$.  
 \end{construction}
\begin{proof}[Proof of Theorem~\ref{theorem: Dirac lower bound}]
We prove first of all that the graph $G_3$ given in Construction~\ref{construction: refined dead-end} above does not contain a transversal $C_r$. Indeed, this can be seen via a simple analysis: suppose for a contradiction that we had chosen $v_i\in V_i$ for $i\in [r]$, and that $G:=G_3[\{v_1, v_2, \ldots, \ v_r\}]$ contains a copy $C$ of $C_r$.

Note that, for all $i\in [r]$, the vertex $v_i$ is connected to at least two other parts in $G$. Otherwise, $v_i$ cannot be part of a transversal cycle. This, together with the fact that vertex $r\in V_{r-1}$ is only connected to $V_r$, implies  $v_{r-1}=0$. Further, since a Hamilton cycle is $2$-connected, $G_3[\{v_1, v_2, \ldots, \ v_{r-1}\}]$ must be connected, which implies that $v_i=0$ for all $i\in [r-2]$(since $v_{r-1}=0$ and edges between parts $V_i$ and $V_j$ for $i<j<r$ exist only between vertices with the same labels).


To complete the cycle $C$, we then need distinct $i,j\in [r-1]$ such that $v_i$ and $v_j$ have a common neighbour in $V_r$. However, for all $i\in [r-1]$, the vertex $0\in V_i$ is only connected to the vertex $i\in V_r$, a contradiction. 


We have thus shown that $G_3$ has $C_r$-free transversals. It now remains to show that we can achieve $d_{K_r}(G_3)>1/2$ for judicious choices of $p_1, p_2,p_3$. This will be achieved by considering a simple optimisation problem. For $i<j$, the edge density between $V_i$ and $V_j$ in $G_3$ is equal to
\begin{itemize} 
		\item $(p_1)^2+(1-p_1)^2$ if $1\leq i<j\leq r-2$;
	\item$p_1p_2$ if $1\leq i\leq r-2$ and $j=r-1$;
	\item $(1-p_1)+(1-p_3)/(r-2)$ if $1\leq i\leq r-2$ and $j=r$;
	\item $(1-p_2)+p_2p_3$ if $i=r-1$ and $j=r$.
	\end{itemize}
We must pick $p_1, p_2,p_3$ to ensure the minimum of these four quantities is strictly greater than $1/2$. This is a simpler task than solving the optimisation problem of maximising the minimum of these quantities. Fix $\varepsilon$ with $0<\varepsilon< \frac{1}{2(r-1)}$. Pick $p_1=\frac{1}{2}+\varepsilon$, $p_2=1-\varepsilon$ and $p_3=1 -(r-1)\varepsilon$ (note all three of these quantities are in $(0,1)$ by our choice of $\varepsilon$). Then for these choices of the parameters $p_1,p_2,p_3$, we have
\begin{align*}
d_{K_r}(G_3)&=\min\left((p_1)^2+(1-p_1)^2, p_1p_2, (1-p_1)+\frac{(1-p_3)}{r-2},(1-p_2)+p_2p_3 \right)\\
&=\min\Bigl(\frac{1}{2}+\varepsilon^2, \frac{1}{2}+\frac{\varepsilon}{2}(1-2\varepsilon), \frac{1}{2}+\frac{\varepsilon}{r-2},\frac{1}{2}+\left(\frac{1}{2}-(r-1)\varepsilon\right) + (r-1)\varepsilon^2\Bigr),
\end{align*}
which is strictly greater than $\frac{1}{2}$, as required. This shows that $\pi_r(C_r)>\frac{1}{2}$, as claimed.

	In fact, we can explicitly solve the optimisation problem needed to work out the best lower bound on $\pi_{K_r}(C_r)$ we can get from Construction~\ref{construction: refined dead-end}. Set $p_1=p_{\star}$, where $p_{\star}$ is the unique real solution in $(\frac{1}{2}, 1)$ to the cubic equation
	\begin{align}\label{eq: optimisation}
	(r-2)-(4r-10)p+(6r-14)p^2-(4r-8)p^3=0,
	\end{align}
	and let
	\begin{align*}
	p_2&=\frac{(p_{\star})^2+(1-p_{\star})^2}{p_{\star}} && p_3=1 - (r-2)p_{\star} (2p_{\star}-1).
	\end{align*}
	Then for these choices of parameters, we have $d_{K_r}(G_3)=(p_{\star})^2+(1-p_{\star})^2$.

\end{proof}

\subsection{Transversal squares in $4$-partite graphs}

\begin{proof}[Proof of Proposition~\ref{prop: C4 bound}]
The lower bound follows from the case $r=4$ Theorem~\ref{theorem: Dirac lower bound}, solving the cubic equation~\eqref{eq: cubic} explicitly. For the upper bound, suppose $G$ is a weighted $4$-partite-graph with $d_{K_4}(G)=p>1/\sqrt{3}$. Let $\sqcup_{i=1}^4V_i$ be the canonical partition of $G$. Select vertices $\mathbf{v_i}\in V_i$ for $i\in[4]$ independently at random, with $v_i=v$ with probability $w(v)$ for every $v\in V_i$. 

Observe that the edges of $K_4$ may be decomposed into $3$ perfect matchings, $M_1$, $M_2$ and $M_3$, the union of any two of which gives a copy of $C_4$. Now the expected number of $M_j$, $1\leq j\leq 3$, such that both of the edges in $M_j$ are present in $G[\{v_1,v_2, v_3, v_4\}]$ is $3p^2>1$. It follows from Markov's inequality that with probability at least $(3p^2-1)/2>0$, $G[\{v_1,v_2, v_3, v_4\}]$ contains a $C_4$. Thus $G$ fails to have $C_4$-free transversals, as required.

\end{proof}
\section{Small cases: connectivity threshold for $H=K_4-P_3$ and $H=C_5$}\label{section: small cases}

\begin{proposition}\label{prop: triangle with a pendant edge}
	$\pi_{K_4-P_3}(\mathcal{T}{4})=4 -2\sqrt{3}$.
\end{proposition}	
\begin{proof}
	For the lower  bound, consider the following construction.  We view $H=K_4-P_3$ as a triangle on a vertex set $\{1,2,3\}$ with a pendant edge $\{0,1\}$ attached to the vertex $1$. We construct an $H$-partite graph by letting $V_0=\{v_0\}$, $V_1=\{v_2, v_3, v_X\}$ with the weighting $w(v_2)=w(v_3)= 2-\sqrt{3}$ and $w(v_X)= 2\sqrt{3}-3$, and let $V_2$, $V_3$ be two disjoint copies of $\{v_1, v_Y\}$ with the weighting $w(v_1)=2-\sqrt{3}$ and $w(v_Y)=\sqrt{3}-1$. We then add the edges $v_0v_2$, $v_0v_3$ between $V_0$ and $V_1$, the edges $v_1v_i$, $v_1v_X$ and $v_Yv_X$ between $V_i$ and $V_1$ ($i\in \{2,3\} )$ and the edge $v_Yv_Y$ between $V_2$ and $V_3$. It is easily checked that the resulting $H$-partite graph has $H$-partite density $4-2\sqrt{3}$ and has no connected transversal.

	For the upper bound, it follows from Proposition~\ref{prop: bondyetal} (and possibly replacing some vertices by two clones, each assigned half the weight) that it is enough to consider $H$-partite graphs with $\vert V_0\vert =1$, $\vert V_1\vert = 3$ and $\vert V_2\vert= \vert V_3\vert \leq 2$. Consider such an $H$-partite graph $G$ with $d_H(G)=p>0$ and no connected transversal. By Proposition~\ref{prop: star connectivity}, we may assume $p\leq 2/3$.

	Set $V_1=\{x_1, x_2, x_3\}$, $V_2=\{y_1, y_2\}$ and $V_3=\{z_1, z_2\}$. We know there exists at least one edge from $V_2$ to $V_3$, so without loss of generality we may assume that $y_2z_2\in E(G)$.  Let $U_2\subseteq V_2$ and $U_3\subseteq V_3$ be the set of vertices in $V_2\cup V_3$ incident with an edge from $V_2$ to $V_3$. Further set $W_1\subseteq V_1$ be the (non-empty) set of vertices sending an edge to $V_0$. Since $G$ contains no connected transversal, there is no edge of $G$ from $W_1$ to $U_2\cup U_3$.

Given  a set $S\subseteq V(G)$, let $w(S):= \sum_{v\in S}w(v)$ be the total weight of the vertices in $S$. By the partite density condition between $V_0$ and $V_1$ we have $w(W_1)\geq p$, and similarly by the partite density condition between $V_2$ and $V_3$ we have $w(U_2)w(U_3)\geq p$. Further by the partite density conditions between $V_1$ and $V_2$ we have $1-p\geq w(W_1)w(U_2)$, and similarly we have $1-p\geq w(W_1)w(U_3)$. We deal with two special cases to show we can restrict our attention to a graph $G$ with a similar structure to our lower bound construction (albeit with potentially different weights).

	\noindent \textbf{Case 1: $W_1=V_1$.} Then our inequalities tell us $1-p\geq \max\left(w(U_2), w(U_3)\right)$, and $w(U_2)w(U_3)\geq p$, so that $(1-p)^2\geq p$ and hence 
	$p\leq\frac{\sqrt{3}-5}{2}<\frac{1}{2}<4-2\sqrt{3}$. Thus moving forward, we may assume that $\vert W_1\vert< \vert V_1\vert$.

	\noindent \textbf{Case 2: $W_1$ sends edges to at most one of $V_2$, $V_3$.} Relabelling parts, we may assume without loss of generality that $W_1$ sends no edges to $V_3$. Then looking at the partite densities between $V_1$ and $V_3$, we have $w(W_1)\leq 1-p$. Now $w(W_1)\geq p$, as observed above, because of the partite density between $V_0$ and $V_1$. This immediately implies $p\leq\frac{1}{2}< 4-2\sqrt{3}$. Thus moving forward, we may assume that $W_1$ sends edges to both of $V_2$ and $V_3$.

	\noindent \textbf{The final case.} Since $G$ does not contain a connected transversal, each vertex of $W_1$ can send an edge to at most one of $V_2$ and $V_3$. Thanks to the previous case, we already know $\vert W_1\vert <\vert V_1\vert =3$. Thus in the remainder of the proof, we may assume without loss of generality all of the following hold: $W_1=\{x_2, x_3\}$, $U_2=\{y_2\}$, $U_3=\{z_2\}$ and $\{x_2y_1,x_3z_1\}\subseteq  E(G)$. By definition of $W_1$, $U_2$ and $U_3$, we then have the following inequalities:
	\begin{align*}
		p&\leq w(x_2)+w(x_3) && \textrm{(density between $V_0$ and $V_1$),}\\
		p & \leq w(x_2)(1-w(y_2))+ (1-w(x_2)-w(x_3))&& \textrm{(density between $V_1$ and $V_2$),}\\
				p & \leq w(x_3)(1-w(z_2))+ (1-w(x_2)-w(x_3))&& \textrm{(density between $V_1$ and $V_3$),}\\
						p & \leq w(y_2)w(z_2)&& \textrm{(density between $V_2$ and $V_3$),}
	\end{align*}
with all weights $w(v)$ taking values in $[0,1]$ and $w(x_2)+w(x_3)\leq 1$. We now analyse this system of inequalities and deduce that $p\leq 4-2\sqrt{3}$.

Noting that the right hand-side of the second and third inequalities above are decreasing in $w(y_2), w(z_2)$ while the right-hand side of the fourth inequality is increasing in $w(y_2), w(z_2)$, we may assume that, reducing the weight of $y_2$ or $z_2$ if necessary, we have $w(y_2)w(z_2)=p$.  Further, note that the right hand-side of the second and third inequalities are decreasing in $w(x_2)$ and $w(x_3)$ respectively, while the right hand-side of the first inequality is increasing in $w(x_2)$ and $w(x_3)$. Reducing the weight of $x_2$ and $x_3$ until the first inequality is tight, we see that we may assume that $w(x_2)+w(x_3)=p$. Thus we may eliminate two of our variables and, rearranging terms, rewrite our system of inequalities as:
\begin{align*}
2p-1 & \leq \min\left\{ w(x_2)\left( 1-w(y_2)\right)	,\   \left(p-w(x_2)\right)\left( 1-\frac{p}{w(y_2)}\right)   \right\},
	\end{align*}
with $w(x_2)\in [0, p]$, $w(y_2)\in [p, 1]$. In particular, we have
\begin{align*}
(2p-1)^2 & \leq w(x_2)(p-w(x_2))\left( 1-w(y_2)\right)\left( 1-\frac{p}{w(y_2)}\right)\leq \left(\frac{p}{2}\right)^2 \left(1-\sqrt{p}\right)^2.
\end{align*}
Solving the inequality $2p-1\leq \frac{p}{2}(1-\sqrt{p})$ for $p\in [0,\frac{2}{3}]$, we deduce that $d_{H}(G)=p\leq 4-2\sqrt{3}$, as required.
\end{proof}

\begin{proposition}\label{prop: C5}
	$\pi_{C_5}(\mathcal{T}_5)=\frac{1}{2}$.
\end{proposition}
\begin{proof}
For the lower bound, we have $\pi_{C_5}(\mathcal{T}_5)\geq \pi_{K_5^-}(\mathcal{T}_{5})=\frac{1}{2}$ by monotonicity and Theorem~\ref{theorem: transversal trees for Kr-matching}, where $K_5^-$ denotes $K_5$ with one edge removed.

For the upper bound, it follows from Proposition~\ref{prop: bondyetal} (and possibly replacing some vertices by two clones, each assigned half the weight) that it is enough to consider $C_5$-partite graphs $G$ with canonical partitions $\sqcup_{i=1}^5 V_i$ satisfying $\vert V_i\vert =\{x_i, y_i\}$ for each $i\in [5]$. Let $G$ be such a graph with no connected transversal and $C_5$-partite density $p$, and suppose for a contradiction that $p> \frac{1}{2}$.

If there exists some $i$ such that neither $x_i$ nor $y_i$ sends an edge into both $V_{i-1}$ and $V_{i+1}$ (winding round modulo $5$), then $p=d(G)\leq \min\left\{w(x_i), w(y_i)\right\}\leq \frac{1}{2}$, contradicting our assumption on $p$. Thus, without loss of generality, we may assume that for every $i\in [5]$ the vertex $x_i$ sends an edge into both $V_{i-1}$ and $V_{i+1}$. If some $x_i$ is adjacent to both $x_{i-1}$ and $x_{i+1}$, then we have a transversal tree in $G$. Further if there is no $i$ such that $x_ix_{i+1}$ is an edge of $G$, then $x_1$, $y_2$, $x_3$, $y_4$, $x_5$ induces a transversal path on $5$ vertices, i.e.\ a connected transversal.

We may therefore assume without loss of generality that $x_1x_2$, $y_5x_1$ and $x_2y_3$ are all edges of $G$, and further that at least one of $y_5x_4$, $y_3x_4$ is an edge of $G$. Thus $G$ contains a path on $5$ vertices, i.e.\ a connected transversal, a contradiction. This concludes the proof that 	$\pi_{C_5}(\mathcal{T}_5)\leq \frac{1}{2}$.
\end{proof}

\section{Further open problems}\label{section: conclusion}
The most obvious directions for future research are, of course, to prove Conjectures~\ref{conj: connectivity} and Conjecture~\ref{conj: Dirac} on the (asymptotic) $r$-partite thresholds for connectivity and Hamiltonicity in $r$-partite graphs.  As in~\cite{BSTT06} and~\cite{CsikvariNagy12}, it may be useful for inductive approaches to these conjectures to consider \emph{inhomogeneous} versions of Problem~\ref{problem: r-partite}, and to determine the set of $K_r$-partite density profiles $\mathbf{\alpha}=\left(\alpha_e\right)_{e\in E(K_r)}$ needed to guarantee the existence of a desirable transversal subgraphs. 

\noindent Beyond Conjectures~\ref{conj: connectivity} and~\ref{conj: Dirac}, there are a number of other natural problems we would like to highlight.
\subsection{Component evolution, long paths and large treess}
In this paper, we focus on the problem of finding connected transversals in multipartite graphs. What if instead we looked for transversals containing a connected component of order at least $t$?
\begin{problem}[Extremal component evolution]\label{problem: comp evo}
For each $t: \ 3\leq t\leq r$, determine $\pi_{K_r}(\mathcal{T}_t)$.
\end{problem}
\noindent We provide below a family of constructions giving lower bounds for this problem for various values of $t$.
\begin{construction}[Intersecting palette construction]\label{construction}
	Fix $t\in \mathbb{Z}_{\geq 2}$. For $r\geq \binom{2t-1}{t}$, we construct a weighted $r$-partite graph $G^t$ with $r$-partition $V(G^t):=\sqcup_{i=1}^r V_i$ as follows.

	Let $\sqcup_{X\in [2t-1]^{(t)}}S_X$ be a balanced partition of $[r]$ into $\binom{2t-1}{t}$  sets $S_X$ indexed by the $t$-elements subsets $X\in [2t-1]^{(t)}$. For each part $V_j$, $j\in [r]$ with  $j\in S_X$, we let $V_j$ consist of a copy of $X$, and we let the weight function be constant and equal to $1/t$ over $V_j$. We then add an edge between $x\in V_i$ and $y\in V_j$  ($i\neq j$) if and only if $x=y$. See Figure~\ref{fig-6.2} for the case $r=6$, $t=2$.
\end{construction}

\begin{figure}\label{fig-6.2}
	\centering
	\includegraphics[scale=0.65]{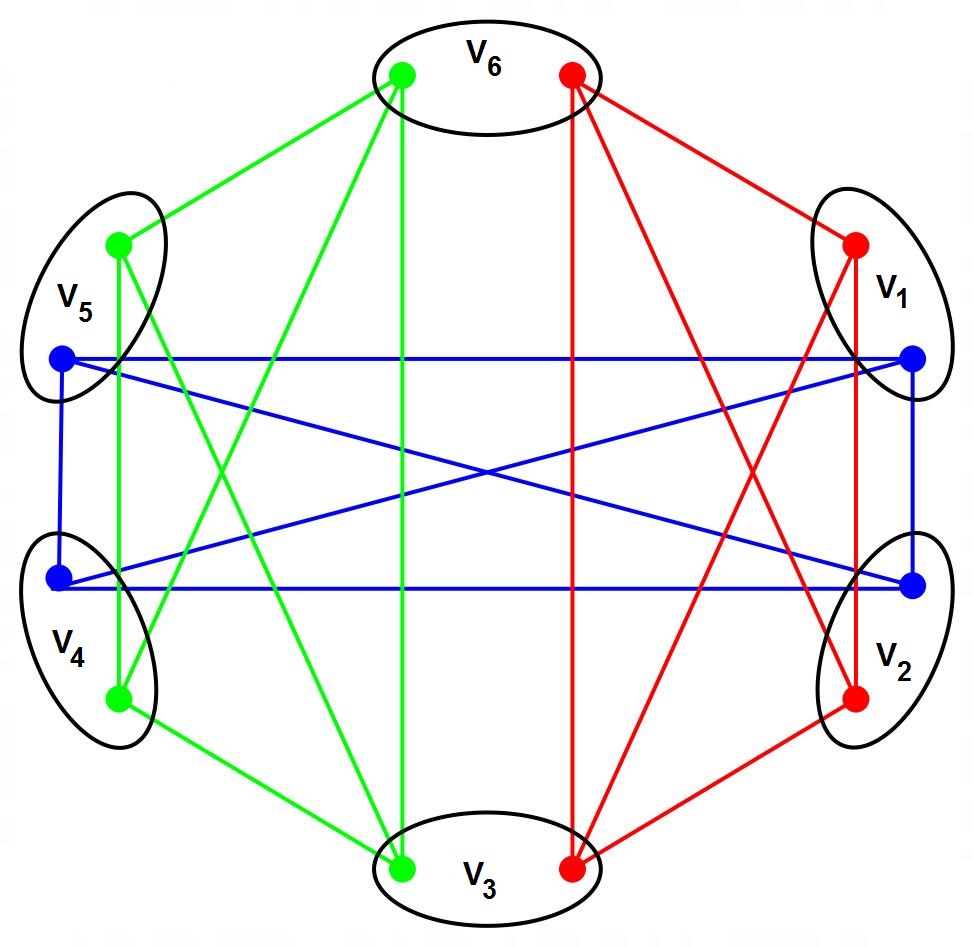}
	\caption{Construction 6.2 for $r=6$ and $t=2$}
\end{figure}	

\begin{proposition}\label{prop: comp evo}
	For every $t\geq 2$ fixed and every $r\geq \binom{2t-1}{t}$, the graph $G^t$	has $r$-partite density $1/t^2$ and contains no transversal connected component on more than $\left\lceil \binom{2t-2}{t-1}/\binom{2t-1}{t}r \right\rceil =\left\lceil\frac{t}{2t-1}r\right\rceil$ vertices.
\end{proposition}
\begin{proof}
Straightforward analysis of Construction~\ref{construction}.
\end{proof}
\noindent Proposition~\ref{prop: comp evo} suggests the following sub-problem of Problem~\ref{problem: comp evo} may be particularly fruitful to study.
\begin{problem}
	Given $\alpha\in (0,1)$, let $f_r(\alpha)$ denote the maximum $r$-partite density of an $r$-partite graph whose transversal components have size at most $\alpha r$. Determine the asymptotic behaviour of $f_r(\alpha)$ as $r\rightarrow \infty$.
\end{problem}

In a different direction, Erd{\H o}s and Gallai determined the largest size of a graph on $n$ vertices containing no path of length $\ell$. It is natural to consider the analogous problem in our setting, and in particular to determine whether long transversal paths are easier to avoid than large transversal trees (i.e.\ whether the answer to the problem below is greater than the answer to Problem~\ref{problem: comp evo} with $t=\ell$).
\begin{problem}
Given $\ell: \ 3\leq \ell \leq r$, determine $\pi_{K_r}(P_{\ell})$.
\end{problem}


\subsection{Bipartite graphs}
Another potentially interesting direction to consider  is the density Tur\'an $H$-partite problem when $H=K_{r,r}$. Clearly, the connectivity threshold in that setting is at least $\pi_{K_{2r}^-}=1/2$ (by Theorem~\ref{theorem: transversal trees for Kr-matching} and monotonicity). We ask whether this lower bound might be tight (which would imply a significant strengthening of Theorem~\ref{theorem: transversal trees for Kr-matching}):
\begin{question}\label{question: bip conn}
Is	$\pi_{K_{r,r}}(\mathcal{T}_{2r})=\frac{1}{2}$ ?
\end{question}
\noindent Similarly to Pfender, we could also look for transversal copies of smaller complete bipartite subgraphs. 
\begin{problem}[Multipartite Zarankiewicz problem]
	Given $2\leq s\leq r$, determine $\pi_{K_{r,r}}(K_{s,s})$.
\end{problem}
\noindent This problem is of particular interest when $s=2$. Note that by the Erd{\H o}s--Stone--Simonovits theorem, $\pi_{K_{r,r}}(K_{2,2})\rightarrow 0$ as $r\rightarrow \infty$, so the question is about the order of the rate of decay with respect to $r$.

Given the connection to locally dependent percolation theory outlined in Section~\ref{section: connection to 1dep}, it would also be of interest to study the connectivity problem for $Q_d$-partite graphs, where $Q_d$ denotes the $d$-dimensional hypercube graph. In particular, we have the following questions:
\begin{question}[Appearance of a transversal giant]\label{question: transversal giant in Qd}
Fix $\varepsilon>0$. For $d$ sufficiently large, what is $\pi_{Q_d}(\mathcal{T}_{\lfloor \varepsilon 2^{d}\rfloor})$?
\end{question}
\begin{question}[Connectivity]\label{question: connectivity in Qd}
	What is $\pi_{Q_d}(\mathcal{T}_{2^{d}})$?
\end{question}
\noindent We note Questions~\ref{question: transversal giant in Qd} and
~\ref{question: connectivity in Qd} are $Q_d$-partite versions of questions of Falgas--Ravry and Pfenninger and of Balister, Johnston, Savery and Scott on $1$-dependent random graphs. In particular,~\cite[Conjecture 1.18]{FRP22} would imply the answer to Question~\ref{question: transversal giant in Qd} is at most $\frac{1}{2}$. We give below a simple matching lower bound.
\begin{proposition}
Fix $d\in \mathbb{N}$ and let $s = 1+\max_{1\leq r\leq d-1} \left(\binom{d}{r-1}+\binom{d}{r}+\binom{d}{r+1}\right)$. Then $\pi_{Q_d}\left(\mathcal{T}_s\right)\geq \frac{1}{2}$. In particular, for all $\varepsilon>0$ fixed and $d$ sufficiently large, the answer to Question~\ref{question: connectivity in Qd} is at least $1/2$. 
\end{proposition}
\begin{proof}
For each $\mathbf{x}=(x_1,x_2, \ldots, x_d)\in \{0,1\}^d$, we let $V_{\mathbf{x}}$ be given by 
\[V_{\mathbf{x}} =\left\{\begin{array}{lll}\{0\} & \textrm{ with weight function }w(\{0\})=1& \textrm{if }\sum_{i=1}^dx_i \equiv 0 \mod 4\\
\{0,1\} & \textrm{ with }w(\{0\})=w(\{1\})=\frac{1}{2}& \textrm{if }\sum_{i=1}^dx_i \equiv 1 \mod 2\\
\{1\} & \textrm{ with }w(\{1\})=1& \textrm{if }\sum_{i=1}^dx_i\equiv 2 \mod 4,
\end{array}\right. \]
with the sets $V_{\mathbf{x}}$ chosen vertex-disjoint. Given an edge $\mathbf{x}\mathbf{y}\in E(Q_d)$, we place an edge between $u\in V_{\mathbf{x}}$ and $v\in V_{\mathbf{y}}$ if and only $u=v$. This is easily seen to give rise to a $Q_d$-partite graph with $Q_d$-partite density $\frac{1}{2}$ and such that every connected component in a transversal subgraph must be contained within the union of at most three consecutive layers $L_t:=\{\mathbf{x}\in Q_d: \ \sum_{i=1}^dx_i=t\}$. Since the size of three consecutive layers is at most $s-1=O(2^d/\sqrt{d})=o(d)$, the proposition follows immediately.		
\end{proof}


\subsection{Cycles}
\noindent Using results of Nagy, we can obtain the following bounds on the threshold for connectivity in $C_r$-partite graphs:
\begin{proposition}\label{prop: cycles}
	For any $r\in \mathbb{Z}_{\geq 4}$, we have:
	\begin{align*}
	\frac{3-\tan^2\left(\frac{\pi}{\lfloor\frac{r}{2}\rfloor+2}\right)}{4}	
	&\leq \pi_{C_{r}}(\mathcal{T}_{r})=\pi_{C_{r}}(P_{r}) 	\leq \frac{3-\tan^2\left(\frac{\pi}{r+2}\right)}{4}.
	\end{align*}
\end{proposition}
\begin{proof}
For the upper bound, we have
\[\pi_{C_{r}}(\mathcal{T}_r)\leq \pi_{C_{r}}(C_r)=\pi_{P_{r+1}}(P_{r+1}) =\frac{3-\tan^2\left(\frac{\pi}{r+2}\right)}{4},\]
where the last two equalities follow from~\cite[Theorem 4.6]{Nagy11}. For the lower bound, view $C_r$ as the union of paths $P^1$ and $P^2$ on $\lceil r/2\rceil +1$ and $\lfloor r/2\rfloor +1$ vertices respectively, with the left and right endpoints of $P^1$ identified with the left and right endpoints of $P_2$. Now consider a $P^2$-partite graph $G^2$ with no transversal copy of $P^2$ and such that $d_{P^2}(G^2)= \pi_{P_{\lfloor r/2\rfloor +1}}(P_{\lfloor r/2\rfloor +1})$. Since the length of $P^2$ is less or equal to the length of $G^1$, we can clearly use $G^2$ to make a $P^1$-partite graph $G^1$ with no transversal copy of $P^1$ and the same partite density as $G^2$. Then the union $G$ of $G^1$ and $G^2$ (identifying the parts corresponding to the endpoints of $P^1$ and $P^2$ as appropriate) is a $C_r$-partite graph in which every transversal has at least two edges missing, and in particular is not connected. Thus $\pi_{C_r}(\mathcal{T}_r)\geq \pi_{P^2}(P^2)$, and the lower bound then follows from~\cite[Corollary 3.13]{Nagy11}.
\end{proof}
\noindent By Theorem~\ref{theorem: transversal trees for Kr-matching} (in the case $r=4$) and Proposition~\ref{prop: C5} (in the case $r=5$), the lower bound in Proposition~\ref{prop: cycles} is tight for $r\in \{4,5\}$.
\begin{question}
Is the lower bound in Proposition~\ref{prop: cycles} tight for all $r\geq 4$?
\end{question}

In a different direction, one could ask for the $K_r$-partite density needed to force the existence of a cycle of a given length $\ell$ (or of length at least/at most $\ell$) as a subgraph of a transversal. This gives rise to the following family of problems
\begin{problem}[Extremal girth and circumference]
For every $\ell$: $3\leq \ell \leq r$, determine $\pi_{Kr}(C_{\ell})$, $\pi_{K_r}\left(\{C_t: \ t\leq \ell\}\right)$ and $\pi_{K_r}\left(\{C_t: \ t\geq \ell\}\right)$.
\end{problem}


\subsection{Spectrum of the connectivity threshold}
For any $r$, let $\mathcal{D}_r:=\{\pi_H(\mathcal{T}_r):\ H \textrm{ is a connected $r$-vertex graph}\}$ denote the collection of thresholds for the existence of connected spanning transversals for connected $r$-vertex graphs.
\begin{problem}
	Characterize $\mathcal{D}_r$.
\end{problem}
\noindent We know that for all connected non-complete graphs $H$ on $r$ vertices, we have
\begin{align*}
\frac{1}{2}=\pi_{K_r^-}(\mathcal{T}_r)\leq \pi_H(\mathcal{T}_r) \leq \pi_{K_{1,r-1}}(\mathcal{T}_r) =\frac{r-2}{r-1}.
\end{align*}
From Nagy's result, Theorem~\ref{theorem: Nagy} and~\eqref{eq: paths}, we have that $\lim_{r\rightarrow \infty}\pi_{P_r}(\mathcal{T}_r)=\frac{3}{4}$. Further, it follows from~\cite[Theorem 26]{DFRH20} that for the ladder $P_r\times K_2$, we have $\lim_{r\rightarrow \infty}\pi_{P_r\times K_2}(\mathcal{T}_{2r})=\frac{2}{3}$. Thus $\frac{1}{2}, \frac{2}{3}, \frac{3}{4}$ and $1$ are accumulation points for the sequence $\mathcal{D}_r$ in the sense that we can find sequences of $r$-vertex connected graphs $H_r$ with $\pi_{H_r}(\mathcal{T}_r)$ tending to these numbers. 
\begin{question}
	What are the other accumulation points for the sequence of finite sets $\mathcal{D}_r$ ?
\end{question}

\subsection{Factors}
Besides spanning trees and Hamilton cycles, another widely-studied class of spanning subgraphs is that of \emph{$F$-factors}: given a graph $F$ and a graph $H$ such that $\vert V(H)\vert \equiv 0 \mod \vert V(F)\vert$, an $F$-factor in $H$ is a collection of vertex-disjoint copies of $F$ that together cover all the vertices of $H$. We denote by $rF$ such a collection of $r$ vertex-disjoint copies of $F$. What is the $r$-partite threshold for a $K_s$-factor?
\begin{problem}[Multipartite Hajnal--Szemer\'edi]\label{problem: hajnal sz}
For $r\in \mathbb{Z}_{\geq 2}$ and $s\in \mathbb{Z}_{\geq 3}$, determine $\pi_{K_{rs}}(rK_s)$.
\end{problem}
\begin{remark}
Problem~\ref{problem: hajnal sz} is very different from the analogous problem for graphs, where one asks for the minimum degree condition for the existence of a $K_s$-factor. Indeed, by the Hajnal--Szemer\'edi theorem, a minimum degree of $2n$ is necessary to force the existence of a $K_3$-factor in a graph	on $3n$ vertices. On the other hand, it follows from the Bondy--Shen--Thomass\'e--Thomassen theorem~\cite{BSTT06} and partitioning a $3n$-partite graph into $n$ tripartite graphs that for any $n$, \[\pi_{K_{3n}}(nK_3)\leq \pi_{K_3}(K_3)=\frac{-1+\sqrt{5}}{2}<\frac{2}{3}.\]
\end{remark}
\noindent Similarly, one could ask for cycle factors.
\begin{problem}[Multipartite Corradi--Hajnal]
	Given $\ell\in \mathbb{Z}_{\geq 3}$, determine   $\pi_{K_{\ell r}}(rC_{\ell})$.
\end{problem}
\noindent Partitioning an $\ell r$-partite graphs into $r$ disjoint $\ell$-partite graphs, we immediately have that   $\pi_{K_{\ell r}}(rC_{\ell})\leq \pi_{K_{\ell}}(C_{\ell})$, which is strictly greater than $\frac{1}{2}$ for all $\ell \geq 4$ (by Theorem~\ref{theorem: Dirac lower bound}) and which we conjecture tends to $1/2$ as $\ell\rightarrow \infty$ (Conjecture~\ref{conj: Dirac}). This should be compared with Abbasi's result~\cite[Chapter 6]{Abbasi} that any $r\ell$-vertex graph with minimum degree at least $\frac{r}{\ell}\left\lfloor\frac{\ell}{2}\right\rfloor$ contains a $C_\ell$-factor.


\subsection{Universality}
The Erd{\H o}s--S\'os conjecture states that every $n$-vertex graph with strictly more than $\frac{n}{2}(k-1)$ edges contains every tree on at most $k$ vertices as a subgraph. Such universality questions are natural in the multipartite setting also:
\begin{question}
Let $3\leq s\leq r$ be fixed. What is the threshold
\[\alpha_{\textrm{tree-universal}}(r,s):=\inf\left\{\alpha>0: \ d_{K_r}(G)\geq \alpha \ \Rightarrow \ \forall T\in \mathcal{T}_s, \ T \textrm{ is a subgraph of a transversal of }G\right\}?\]
\end{question}
\noindent By Proposition~\ref{prop: general connectivity upper bound}, it is immediate that $\alpha_{\textrm{tree-universal}}(r,s) \leq\frac{s-2}{s-1}$. However, we suspect that the actual threshold may be significantly lower when $r$ is larger than $s$.

\subsection*{Acknowledgements}
This research was carried out while the first author was a long-term visitor at the Department of Mathematics and Mathematical Statistics at Ume{\aa} University, whose hospitality is gratefully acknowledged.

\end{document}